\documentclass[11pt,renqo,reqno]{amsart}

\usepackage{graphicx} 
\usepackage{verbatim} 
\usepackage{amsmath, amsthm, amssymb, amsrefs} 
\usepackage{anysize} 
\usepackage[autostyle,english=american]{csquotes} 
\usepackage[shortlabels]{enumitem} 
\usepackage{hyperref} 
\usepackage{bm} 
\usepackage{color} 
\usepackage[parfill]{parskip} 

\marginsize{2.25cm}{2.25cm}{2.25cm}{2.25cm} 
\MakeOuterQuote{"} 
\setlist[enumerate]{topsep=0pt,itemsep=0ex,partopsep=1ex,parsep=1ex,leftmargin=*} 
\numberwithin{equation}{section} 

\newtheorem{theorem}{Theorem}[section]
\newtheorem{cor}[theorem]{Corollary}
\newtheorem{lemma}[theorem]{Lemma}
\newtheorem{prop}[theorem]{Proposition}
\theoremstyle{remark}

\theoremstyle{definition}

\let\vec\bm

\def\to{\rightarrow}

\def\R{\mathbb{R}}

\def\1{1\!\!1}

\title{Perfectly packing a cube by cubes of nearly harmonic sidelength}
\thanks{Research of the author is partially supported by NSERC CGS-M}
\author{Rory McClenagan}
\address{Department of Mathematics and Statistics \\
        University of Northern British Columbia \\
        Prince George, BC V2N4Z9 \\
        Canada}
\date{\today}

\begin{document}
\maketitle
\begin{abstract}
Let $d$ be an integer greater than $1$, and let $t$ be fixed such that $\frac{1}{d} < t < \frac{1}{d-1}$. We prove that for any $n_0$ chosen sufficiently large depending upon $t$, the $d$-dimensional cubes of sidelength $n^{-t}$ for $n \geq n_0$ can perfectly pack a cube of volume $\sum_{n=n_0}^\infty \frac{1}{n^{dt}}$. Our work improves upon a previously known result in the three-dimensional case for when $1/3 < t \leq 4/11 $ and $n_0 = 1$ and builds upon recent work of Terence Tao in the two-dimensional case.
\end{abstract}

\section{Introduction}
Let $d$ be an integer greater than $1$. We define a \textit{brick} to be a closed $d$-dimensional hyperrectangle and use the term \textit{cube} to refer to a brick with equal sidelengths. We define a \textit{packing} of a finite or infinite collection of bricks $\mathcal{B}$ to be a particular configuration of the bricks in $\R^d$ such that the interior of the bricks are disjoint and the facets of the bricks are parallel to the coordinate hyperplanes. A \textit{packing of $\mathcal{B}$ in a solid $\Omega \subset \R^d$} is a packing of $\mathcal{B}$ such that every brick is contained in $\Omega$. The packing is \textit{perfect} if the measure $m(\Omega \setminus \mathcal{B})$ is $0$. In this case the sum of the volumes of the bricks must be equal to $m(\Omega)$.

A famous question posed by Meir and Moser \cite{meir_moser-original} asks  whether rectangles of dimensions $\frac{1}{n}\times \frac{1}{n+1}$ for $n \geq 1$ can perfectly pack a square of area $1$. They also ask whether squares of dimensions $\frac{1}{n} \times \frac{1}{n}$ for $n \geq 2$ can perfectly pack a square of area $\frac{\pi^2}{6} - 1$. While both of these problems remain open, there are two directions in which partial results have been obtained.

First, one can instead try to pack the same squares into a slightly larger square. For instance, Paulhus \cite{paulhus-squares} showed that the squares of sidelength $1/n$ for $n \geq 2$ could be packed into a square of area $\frac{\pi^2}{6} - 1 + \frac{1}{1244918662}$. However, it was pointed out in \cite{joos-rectangles} that the proof contained some errors; these errors were corrected in \cite{januszewski-correction}.

Second, one can instead consider the problem of trying to perfectly pack the squares of sidelength $1/n^t$ for $n \geq 1$ and some fixed $t > 1/2$ into a square of area $\sum_{n=1}^\infty \frac{1}{n^{2t}}$. This becomes harder as $t \to 1^-$, and is obviously equivalent to the original problem when $t = 1$. Januszewski and Zielonka \cite{januszewski-pack} verified this for $1/2 < t \leq 2/3$. At the expense of dropping the first few squares, Tao \cite{tao-pack} has recently proved that one could perfectly pack the collection of squares of sidelength $1/n^t$ for $n \geq n_0$ (where $n_0$ depends on $t$) into a square of area $\sum_{n=n_0}^\infty \frac{1}{n^{2t}}$ in the entire range $1/2 < t < 1$.

The analogous problem for cubes in $\R^3$ has also been studied. For instance, building on the methods presented in \cite{joos-cubes}, Januszewski and Zielonka \cite{januszewski-pack} have shown that cubes of sidelength $1/n^t$ for $n \geq 1$ can be packed perfectly into a cube of volume $\sum_{n=1}^\infty \frac{1}{n^{3t}}$ provided that $1/3 < t \leq 4/11$. In the general $d$-dimensional case, Jo\'os showed in \cite{joos-d-cubes} that $d$-cubes of sidelength $n^{-t}$ for $n \geq 1$ can be packed into a cube of volume $\zeta(dt)$ as long as $t$ is in the interval $[d_0,2^{d-1}/(d2^{d-1}-1)]$, where $d_0$ depends only on $d$. The lower bound, $d_0$, which Jo\'os implicitly defined, was later improved by Januszewski and Zielonka in \cite{januszewski-d-cubes} to $1/d$.

In this paper, we extend Tao's work \cite{tao-pack} in the $2$-dimensional case to the $d$-dimensional case of cubes and prove the following result:
\begin{theorem}\label{r-main}
If $\frac{1}{d} < t < \frac{1}{d-1}$, and $n_0$ is sufficiently large depending on $t$, then the cubes of sidelength $n^{-t}$ for $n \geq n_0$ can perfectly pack a cube of volume $\sum_{n=n_0}^\infty \frac{1}{n^{dt}}$.
\end{theorem}

To prove Theorem \ref{r-main}, we apply an inductive-type argument similar to that used by Tao in \cite{tao-pack}. Initially, we suppose that we can pack a finite set of cubes $\mathcal{C}$ of sidelength $n^{-t}$ for $n_0 \leq n < n_0'$ into our single cube $S$. As long as $S \setminus \mathcal{C}$ can be partitioned into bricks $\mathcal{B}$ with small enough total surface area, then  we can find a brick $B\in\mathcal{B}$ which is wide enough to pack the next cube of sidelength $(n_0')^{-t}$. We pack $B$ by cubes $\mathcal{C}'$ of sidelength $n^{-t}$ for $n_0'\leq n < n_{0}^{''}$ in some efficient manner. By efficient, we mean that the remaining space $B\setminus\mathcal{C}'$ can be partitioned into bricks $\mathcal{B}'$ with small enough total surface area. In the next iteration, we choose a wide brick from $\mathcal{B} \setminus \{B\} \cup \mathcal{B}'$ and pack it efficiently. We proceed recursively until we have packed an arbitrarily large finite number of cubes into $S$. Theorem \ref{r-main} would then follow from a compactness argument. 

This type of argument reduces the problem to finding a general technique for packing cubes efficiently into some brick, in essence, forming the inductive step in the above argument. Up until now, we have followed Tao's argument in \cite{tao-pack} closely. However, while it is fairly straightforward to generalize the standard two-dimensional packing algorithm used in \cite{tao-pack} to the higher-dimensional case, Tao's method of \textit{explicitly} verifying that this packing is legal and efficient becomes much more difficult in three dimensions or more. Our innovation is to introduce the notion of "snugness" (see Section \ref{s-lemmas}); this allows us to perform this portion of the argument in an elegant fashion which does not become too complex in the higher-dimensional setting.

In Section \ref{s-lemmas}, we introduce our notation and prove some simple lemmas. In Section \ref{s-reductions}, we reduce the proof of Theorem \ref{r-main} to a more general result, Proposition \ref{r-ind_pack}, which can be proved via induction. The inductive step of this argument is furnished by Theorem \ref{r-pack} which provides a general and efficient method for packing a brick by cubes. This result is proved in Section \ref{s-pack}.

\subsection*{Acknowledgements}
We would like to thank Professor Terence Tao for the discussion on his blog post describing his work in \cite{tao-pack} and for answering our questions. We would also like to thank Dr.~Alia Hamieh for helpful comments and proof-reading this paper, as well as Jaume de Dios, Dr. Rachel Greenfeld, and Dr. Jose Madrid for helpful comments. Finally, we would like to thank the referee for useful suggestions.

\section{Preliminary Lemmas and Notation}\label{s-lemmas}
Throughout this paper, we will use the standard asymptotic notation $X= O(Y)$, $X \ll Y$, and $Y \gg X$ to refer to the relation $X \leq C |Y|$. The constant $C$ will only be allowed to depend on the parameters $t$ and $\delta$, which we will introduce shortly, but will be independent of all other parameters. If, instead, we use the notation $X = O_M(Y)$ or $X \ll_M Y$, then the corresponding constant $C$ is allowed to depend on the parameter $M$. We use $X \asymp Y$ if $X \ll Y$ and $Y\ll X$. Finally, we use the notation $X = o(Y)$ if $X/ Y \to 0$ with respect to some explicit or implicit limiting behaviour defined in context. We will implicitly assume that all of our constants are allowed to depend upon the dimension $d$.

Note that we also use some non-standard asymptotic notation. If $\vec{x} = (x_1,x_2, \dots, x_d) \in \R^d$, then we use $\vec{x} + O(X)$ to refer to a vector $(x_1 + O(X), \dots , x_d+ O(X))$, and analogously for little-$o$ notation. Similarly, if $B= [B_1, B_1'] \times \dots \times [ B_d, B_d']$ is a brick positioned in $\R^d$, then we use $B+ O(X)$ to refer to a brick $[B_1 + O(X), B_1' + O(X)] \times \dots \times [ B_d + O(X), B_d' + O(X)]$, and analogously for little-$o$ notation.

Let $i,j \in \{1, 2, \dots, d\}$. Given a brick $B$, we will denote its sidelengths by $w_i(B)$, ordered so that $w_i(B) \leq w_j(B)$ for any $i \leq j$. We say that the \textit{width} of $B$ is the smallest sidelength, and denote it by $w(B) := w_1(B)$. Clearly, $w_i(B) = w_j(B)$ for every $i$ and $j$ if and only if $B$ is a cube. We define the \textit{volume} of a single brick $B$ to be
    \[\mathrm{vol}(B) := w_1(B) w_2(B) \dots w_d(B).\]
We define the \textit{eccentricity} of a brick as
    \[\mathrm{ecc}(B) := \frac{\mathrm{vol}(B)}{w(B)^d} \geq 1.\]
Note that $\mathrm{ecc}(B) = 1$ if and only if $B$ is a cube. 

Let $\mathcal{B}$ be a collection of bricks. Define the \textit{volume} of $\mathcal{B}$ to be
    \[\mathrm{vol}(\mathcal{B}) := \sum_{B \in \mathcal{B}} w_1(B) w_2(B) \dots w_d(B).\]
Define the \textit{unweighted surface area} of $\mathcal{B}$ to be
    \[\mathrm{surf}(\mathcal{B}) := 2\sum_{B \in \mathcal{B}} \sum_{1 \leq i_1< i_2 < \dots < i_{d-1} \leq d} w_{i_1}(B) w_{i_2}(B) \dots w_{i_{d-1}}(B) \asymp \sum_{B \in \mathcal{B}} w_2(B) w_3(B) \dots w_d(B).\]
For $0 \leq \delta < 1$, define the \textit{weighted surface area} of $\mathcal{B}$ to be
    \[\mathrm{surf}_\delta(\mathcal{B}) := \sum_{B \in \mathcal{B}} w_1(B)^\delta w_2(B) \dots w_d(B).\]
Clearly, $\mathrm{surf}_0(\mathcal{B}) \asymp \mathrm{surf}(\mathcal{B})$. Weighted surface area is roughly speaking a version of unweighted surface area which weights high eccentricity bricks a little less than low eccentricity bricks. We can use the inequality $w(B) \leq (w_2(B) w_3(B) \dots w_d(B) )^{\frac{1}{d-1}}$ for any brick $B$, to derive the crude bound
\begin{equation}\label{gen-surf_relation}
    \mathrm{surf}_\delta (\mathcal{B}) \ll (\mathrm{surf}(\mathcal{B})) ^{1+ \frac{\delta}{d-1}},
\end{equation}
for a finite collection of bricks $\mathcal{B}$.

A solid $S \subset \R^d$ is called \textit{simple} if it is connected and can be formed as a union of a finite collection of bricks. A packing of a finite collection of bricks $\mathcal{B}$ in a simple solid $S$ is called \textit{$\varepsilon$-snug} for some $\varepsilon > 0$ if, for every brick $B \in \mathcal{B}$, the portion of $\partial B$ which does not intersect the boundary of another brick or the boundary of $S$ has surface area $\ll (w\varepsilon)^{d-1}$ and the portion of $\partial S$ which does not intersect the boundary of any brick in $\mathcal{B}$ also has surface area $\ll (w\varepsilon)^{d-1}$. Here $w$ is the width of the widest brick in $\mathcal{B}$. The \textit{size discrepancy} of a finite collection of bricks $\mathcal{B}$ is
    \[\mathrm{sd}(\mathcal{B}) = \frac{\max_{B \in \mathcal{B}} w(B)}{\min_{B \in \mathcal{B}} w(B)} - 1.\]

The following lemma gives a criterion for the existence of a brick of a certain minimum width in terms of an elegant relationship between volume and weighted surface area.
\begin{lemma}\label{r-brick_width}
Let $0 \leq \delta < 1$. For any finite collection of bricks $\mathcal{B}$, there exists a brick with width at least $\big(\frac{\mathrm{vol}(\mathcal{B})}{\mathrm{surf}_\delta(\mathcal{B})}\big)^{\frac{1}{1-\delta}}$.
\end{lemma}

\begin{proof}
By definition,
    \[\mathrm{vol}(\mathcal{B}) = \sum_{B \in \mathcal{B}} w_1(B) w_2(B) \dots  w_d(B) \leq \Big( \sup_{B \in \mathcal{B}} w(B)^{1-\delta} \Big) \mathrm{surf}_\delta(\mathcal{B}).\]
This implies that $(\mathrm{sup}_{B \in \mathcal{B}} w(B))^{1-\delta} \geq \frac{\mathrm{vol}(\mathcal{B})}{\mathrm{surf}_\delta (\mathcal{B})}$, giving the desired result.
\end{proof}
This illustrates the principal behind using weighted surface area. If we were to use unweighted surface area, namely setting $\delta = 0$, then to guarantee the existence of a brick of width $w$, we would need an upper bound on the surface area of the form $w^{-1}\mathrm{vol}(\mathcal{B})$. However, if $\delta > 0$, then we only need a weaker bound of the form $w^{-(1-\delta)} \mathrm{vol}(\mathcal{B})$.

The following result follows from a compactness argument (see, for example, \cite{martin-compactness}).
\begin{lemma}\label{r-compactness}
Let $\mathcal{B}$ be an, at most, countable collection of bricks and let $\Omega \subset \R^d$ be compact. Suppose that an arbitrarily large, but finite, number of bricks from $\mathcal{B}$ can be packed into $\Omega$. Then $\mathcal{B}$ in its entirety can be packed into $\Omega$. Furthermore, if $\mathrm{vol}(\mathcal{B}) = \mathrm{m}(\Omega)$, then this packing is perfect.
\end{lemma}

The following lemma states that if a packing of bricks is sufficiently snug, then the region between the bricks has negligible surface area.

\begin{lemma}\label{r-snug}
Suppose that a finite collection of bricks $\mathcal{B}$, where the widest brick has width $w$, has a $\varepsilon$-snug packing in a brick $B$, for some $\varepsilon > 0$. Then, $B \setminus \mathcal{B}$ can be partitioned into bricks with weighted surface area $ \ll C_{|\mathcal{B}|} \nu w^{d-1+\delta}$, where $C_{|\mathcal{B}|}$ is a constant that depends on $|\mathcal{B}|$ and $\nu \to 0$ as $\varepsilon \to 0$.
\end{lemma}

\begin{proof}
Partition $B \setminus \mathcal{B}$ into a finite number of bricks $\mathcal{B}'$. The maximum number of bricks in $\mathcal{B}'$ can be bounded by a constant dependent upon $|\mathcal{B}|$. By the definition of snugness, we know that the true surface area, $A$, (in the sense of the ($d-1$)-dimensional Lebesgue measure) of the solid $\cup \mathcal{B}'$ could not exceed $(\varepsilon w)^{d-1} (|\mathcal{B}| + 1)$. The result follows from the crude bound $\mathrm{surf} (\mathcal{B'}) \ll A |\mathcal{B}'| $ and (\ref{gen-surf_relation}).
\end{proof}

\section{Initial Reductions}\label{s-reductions}
In this section we prove the higher-dimensional analogue of Proposition 2.1 in \cite{tao-pack} which will allow us to deduce Theorem \ref{r-main}.

\begin{prop}\label{r-ind_pack}
Fix $\frac{1}{d} < t < \frac{1}{d-1}$ and $\delta$ depending on $t$, such that $0 < \delta < 1$ and $(d-1)t + \delta t < 1$. Choose a scale $M$ sufficiently large and choose $N_0$ sufficiently large depending on $M$. Let $n_{\mathrm{max}} \geq n_0 \geq N_0$, and suppose that $\mathcal{B}$ is a family of bricks with volume
\begin{equation}\label{ind_pack-vol}
    \mathrm{vol}(\mathcal{B}) = \sum_{n=n_0}^\infty  \frac{1}{n^{dt}},
\end{equation}
weighted surface area bound
\begin{equation}\label{ind_pack-surf}
    \mathrm{surf}_\delta (\mathcal{B}) \ll \frac{1}{M^{1-\delta/2}} \sum_{n=1}^{n_0-1} \frac{1}{n^{(d-1)t+\delta t}},
\end{equation}
and height bound
\begin{equation}\label{ind_pack-height}
    \sup_{B \in \mathcal{B}} w_d(B) \ll 1.
\end{equation}
Then one can pack $\bigcup_{B \in \mathcal{B}} B$ by cubes of sidelength $n^{-t}$ for $n_0 \leq n < n_{\mathrm{max}}$.
\end{prop}

First we see how we can derive our main result from Proposition \ref{r-ind_pack}.

\begin{proof}[Proof of Theorem \ref{r-main}]
Fix $\delta = \frac{1}{d-1}-t$, and note that it easily satisfies the necessary conditions. Take $\mathcal{B} = \{C\}$ where $C$ is the cube of volume $\sum_{n=n_0}^\infty \frac{1}{n^{dt}}$, having sidelength $\ll n_0^{1/d - t}$ (since $t > 1/d$). Observe that (\ref{ind_pack-surf}) is satisfied, since
    \[\mathrm{surf}_\delta (\mathcal{B}) \ll n_0^{(1/d-t)(d-1+\delta)} \ll n_0^{1-dt} \ll \frac{1}{M^{1-\delta/2}} n_0^{1-(d-1)t-\delta t} \ll \frac{1}{M^{1-\delta/2}} \sum_{n=1}^{n_0-1} \frac{1}{n^{(d-1)t+\delta t}},\]
recalling that $(d-1)t + \delta t < 1$ and $0 < \delta < 1$. We also have used the fact that $\frac{n_0^{t- \delta t}}{ M^{1 - \delta / 2}} \gg 1$ since $n_0 \geq N_0$, which is sufficiently large depending upon $M$. Since (\ref{ind_pack-vol}) and (\ref{ind_pack-height}) are trivially satisfied, we can then apply Proposition \ref{r-ind_pack} to conclude that $C$ can be packed by cubes of sidelength $n^{-t}$ for $n_0 \leq n < n_{\mathrm{max}}$, and the result follows from Lemma \ref{r-compactness}.
\end{proof}

The inductive step in the proof of Proposition \ref{r-ind_pack} requires us to pack a brick by a collection of cubes. We isolate this result as a corollary to the following more general theorem:

\begin{theorem}\label{r-pack}
Fix $0 \leq \delta < 1$. Let $M=M_1 \leq M_2 \leq \dots \leq M_d$ be natural numbers, and $\mathcal{C}$ be a family of $M_* = M_1 M_2 \dots M_d$ cubes with maximum width $w$ and with size discrepancy $\varepsilon$, for some $\varepsilon > 0$. Let $S$ be a brick with dimensions $S_1 \times S_2 \times \dots \times S_d$ satisfying $M_i w \leq S_i \leq M_i w + O(w)$ for $i \in \{1,2, \dots, d\}$. Then, there exists a packing of $\mathcal{C}$ in $S$ such that $S \setminus \mathcal{C}$ can be partitioned into bricks $\mathcal{B}$ satisfying
    \[\mathrm{surf}_\delta (\mathcal{B}) \ll \frac{M_*}{M} w^{d-1+\delta} + C_M \nu w^{d-1+\delta},\]
where $C_M$ is some constant dependent upon $M$ and $\nu \to 0$ as $\varepsilon \to 0$.
\end{theorem}

We will prove this theorem in Section \ref{s-pack}. For now, we use it to derive the corollary we need in the proof of Proposition \ref{r-ind_pack}:

\begin{cor}\label{r-pack_cor}
Fix $\frac{1}{d} < t < \frac{1}{d-1}$ and $\delta$ depending on $t$, such that $0 < \delta < 1$ and $(d-1)t + \delta t < 1$. Choose a scale $M$ sufficiently large and choose $N_0$ sufficiently large depending on $M$. Suppose that $S$ is a brick satisfying the width bound $Mn_0^{-t} \leq w(S) \leq Mn_0^{-t} + O(n_0^{-t})$ for some $n_0 \geq N_0$ and satisfying the eccentricity bound $\mathrm{ecc}(S) = o(n_0)$. Then we can find $n_0' \geq n_0$ with $n_0' - n_0 \asymp \mathrm{ecc}(S) M^d$, such that $S$ can be perfectly packed by cubes of sidelength $n^{-t}$ for $n_0 \leq n < n_0'$ and a collection of bricks $\mathcal{B}$ satisfying the weighted surface area bound
    \[\mathrm{surf}_\delta (\mathcal{B}) \ll \frac{1}{M} \sum_{n=n_0}^{n_0'-1} \frac{1}{n^{(d-1)t+\delta t}} .\]
\end{cor}

\begin{proof}
Let the parameters be chosen as in the corollary, and use the notation $S= S_1 \times S_2 \times \dots \times S_d$ such that $S_1 \leq S_2 \leq \dots \leq S_d$. Thus, $M n_0^{-t} \leq S_1 \leq Mn_0^{-t} + O(n_0^{-t})$. Define $M_i = \lfloor S_i/ n_0^{-t} \rfloor$ for $i \in \{1,2,\dots, d\}$, so that $M_1 \asymp M$. Choose $n_0' = n_0 + M_*$. Note that
\begin{equation}\label{pack_cor-M}
    M_* \asymp \mathrm{ecc}(S) M^d,
\end{equation}
as required. Let $\mathcal{C}$ be the collection of cubes of sidelength $n^{-t}$ for $n_0 \leq n < n_0'$. The size discrepancy is $\frac{n_0'^t}{n_0^t} - 1$. This can be made arbitrarily small as long as $N_0$ is chosen to be sufficiently large compared with $M$. This makes the second term in the bound of Theorem \ref{r-pack} negligible with respect to the first. Thus, we can apply Theorem \ref{r-pack} to get a packing of $\mathcal{C}$ in $S$ such that $S \setminus \mathcal{C}$ can be partitioned into bricks $\mathcal{B}$ satisfying
    \[\mathrm{surf}_\delta(\mathcal{B}) \ll \frac{M_*}{M} (n_0^{-t})^{d-1+\delta} \ll \frac{M_*}{M} \frac{(1 + \mathrm{sd}(\mathcal{C}))^{d-1+\delta}}{(n_0')^{(d-1)t+\delta t}} \ll \frac{1}{M} \sum_{n=n_0}^{n_0'-1} \frac{1}{n^{(d-1)t+\delta t}}, \]
since $\mathrm{sd}(\mathcal{C}) \to 0$. This completes the proof.
\end{proof}

Observe that the power of $M$ in the weighted surface area bound of the corollary is independent of $\delta$. This fact allows us to loosen our weighted surface area bound (\ref{ind_pack-surf}) by a factor of $M^{\delta / 2}$, which is enough to let us complete the inductive step of the proof of Proposition \ref{r-ind_pack}, illustrating the advantage of working with weighted surface area (see also the discussion after Lemma \ref{r-brick_width}).

We now use this corollary to prove Proposition \ref{r-ind_pack}. Our proof closely mirrors the proof of Proposition 2.1 in \cite{tao-pack} except for higher dimensions. However, for the reader's convenience, we include it here.
\begin{proof}[Proof of Proposition \ref{r-ind_pack}]
We prove this via downward induction on $n_0$. Fix $n_{\mathrm{max}}\geq N_0$. Clearly, the result holds if $n_0= n_{\mathrm{max}}$. Fix some $n_0 \leq n_{\mathrm{max}}$, and assume the result holds with $n_0$ replaced by any strictly larger integer up to $n_{\mathrm{max}}$. We show that the result will then hold for $n_0$.

Since $t > 1/d$, (\ref{ind_pack-vol}) implies that $\mathrm{vol}(\mathcal{B}) \asymp n_0^{1-dt}$. Furthermore, since $(d-1)t+ \delta t < 1$, we have $\mathrm{surf}_\delta(\mathcal{B}) \ll M^{-(1-\delta/2)} n_0^{1-(d-1)t-\delta t}$. Thus, Lemma \ref{r-brick_width} implies the existence of a brick $B' \in \mathcal{B}$ satisfying
    \[w(B') \gg \bigg( \frac{n_0^{1-dt}}{M^{-(1-\delta/2)} n_0^{1-(d-1)t-\delta t}} \bigg)^{\frac{1}{1-\delta}} = M^{\frac{1-\delta/2}{1-\delta}} n_0^{-t}.\]
Since $\frac{1-\delta/2}{1-\delta} > 1$ for $0 < \delta < 1$, then as long as we take $M$ sufficiently large, we can drop the implied constant and conclude that $w(B') \geq Mn_0^{-t}$. Partition $B'$ into two bricks $B$ and $B' \setminus B$, so that $M n_0^{-t} \leq w(B) \leq M n_0^{-t} + O(n_0^{-t})$. By the height bound, (\ref{ind_pack-height}), $\mathrm{ecc}(B) \ll M^{-(d-1)} n_0^{(d-1)t} = o( n_0)$, which means that we can apply Corollary \ref{r-pack_cor} to pack $B$ by cubes of sidelengths $n^{-t}$ for $n_0 \leq n < n_0'$ with $n_0' - n_0 \gg M^d$ and a collection of bricks $\mathcal{B}_0$ satisfying
\begin{equation}\label{ind_pack-B_0_surf}
    \mathrm{surf}_\delta (\mathcal{B}_0) \ll \frac{1}{M} \sum_{n=n_0}^{n_0'-1} \frac{1}{n^{(d-1)t+\delta t}} .
\end{equation}
Now, if $n_0' \geq n_{\mathrm{max}}$, then we are done, as we have packed every cube of sidelengths $n^{-t}$ for $n_0 \leq n < n_{\mathrm{max}}$. Otherwise, suppose that $n_0' < n_{\mathrm{max}}$. Since $n_0'$ is strictly larger than $n_0$, it makes sense to now apply our inductive hypothesis, replacing $n_0$ by $n_0'$ (which is strictly larger than $n_0$), and replacing $\mathcal{B}$ by $\mathcal{B}' = (\mathcal{B} \setminus \{B'\}) \cup \{B' \setminus B\} \cup \mathcal{B}_0$. First, however we have to check to assure that the conditions of the proposition are met. Observe,
    \[\mathrm{vol}(\mathcal{B}') = \sum_{n=n_0}^\infty \frac{1}{n^{dt}} - \sum_{n=n_0}^{n_0'-1} \frac{1}{n^{dt}} = \sum_{n=n_0'}^\infty \frac{1}{n^{dt}},\]
and so $\mathcal{B}'$ has the required total volume (\ref{ind_pack-vol}). Clearly, $\mathcal{B}'$ satisfies the height bound (\ref{ind_pack-height}). Finally, by (\ref{ind_pack-B_0_surf}), (\ref{ind_pack-surf}), and the fact that $\mathrm{surf}_\delta \{ B' \setminus B \} \leq \mathrm{surf}_\delta \{ B' \}$, we have
    \[\mathrm{surf}_\delta(\mathcal{B}') \leq \mathrm{surf}_\delta (\mathcal{B} \cup \mathcal{B}_0) \ll \frac{1}{M^{1-\delta/2}} \sum_{n=1}^{n_0-1} \frac{1}{n^{(d-1)t+\delta t}} + \frac{1}{M} \sum_{n=n_0}^{n_0'-1} \frac{1}{n^{(d-1)t+\delta t}} \ll \frac{1}{M^{1-\delta/2}} \sum_{n=1}^{n_0'-1} \frac{1}{n^{(d-1)t+\delta t}},\]
and thus $\mathcal{B}'$ satisfies the weighted surface bound (\ref{ind_pack-surf}). Thus, we can apply the inductive hypothesis for $n_0'$ and pack $\bigcup_{B \in \mathcal{B}'} B$ by the remaining cubes of sidelength $n^{-t}$ for $n_0' \leq n < n_{\mathrm{max}}$, which in turn implies that we can pack $\bigcup_{B \in \mathcal{B}} B$ by cubes of sidelength $n^{-t}$ for $n_0 \leq n < n_{\mathrm{max}}$.
\end{proof}

All that remains is proving  Theorem \ref{r-pack} which provides a general and efficient brick-packing algorithm.

\section{Efficient Brick-Packing Algorithm}\label{s-pack}
\begin{proof}[Proof of Theorem \ref{r-pack}]
Note that, without loss of generality, we can assume that $S_i = M_i w$ for $i \in \{1,2,\dots, d\}$. To see this, suppose that $S'$ is a cube that contains $S$ and instead satisfies $M_iw \leq S_i' \leq M_iw + O(w)$. Then, $S' \setminus S$ can be partitioned into a $O(1)$ bricks, each of which contributes an allowable weighted surface area $\ll \frac{M_*}{M} w^{d-1+\delta}$.

To explicitly define our packing, we position $S$ in $\R^d$ as
    \[[0, M_1w] \times [0, M_2w] \times \dots \times [0, M_d w].\]
Index $\mathcal{C}$ as $\{C_n\}_{n=0}^{M_*-1}$ from largest width to smallest width. We use the notation $w_n := w(C_n)$. By construction, $w_n \leq w_m$ if and only if $n \geq m$. We further index $n=0,1, \dots, M_* - 1$ by $n_{\vec{i}} = n_{i_1,i_2,\dots, i_d}$, where
    \[n_{i_1,i_2,\dots, i_d} := i_1 + i_2 M_1 + i_3 M_1 M_2 + \dots + i_d M_1 M_2 \dots M_{d-1},\]
for $i_k = 0, \dots, M_k-1$ (with $k=1,2,\dots, d$). We use the notation $C_{\vec{i}} := C_{n_{\vec{i}}}$ and $w_{\vec{i}} := w_{n_{\vec{i}}}$. Position each $C_{\vec{i}}$ in $S$ as
    \[C_{\vec{i}} := [x^1_{\vec{i}}, x^1_{\vec{i}} + w_{\vec{i}}] \times [x^2_{\vec{i}}, x^2_{\vec{i}} + w_{\vec{i}}] \times \dots \times [x^d_{\vec{i}}, x^d_{\vec{i}} + w_{\vec{i}}],\]
where for any $k \in \{1,2, \dots, d\}$, we define
    \[x^{k}_{\vec{i}} = \sum_{i_{k}' = 0}^{M_{k} - 1} w_{i_1,\dots, i_{ k- 1}, i_{k}', 0, \dots, 0} - \sum_{i_{k}' = i_{k}} ^{M_{k}-1} w_{i_1,\dots, i_{k - 1},  i_{k}', i_{k} +1, \dots i_d} \]
(see Figure \ref{fig-C_i}).
\begin{figure}[t]
    \includegraphics[width=.6\textwidth]{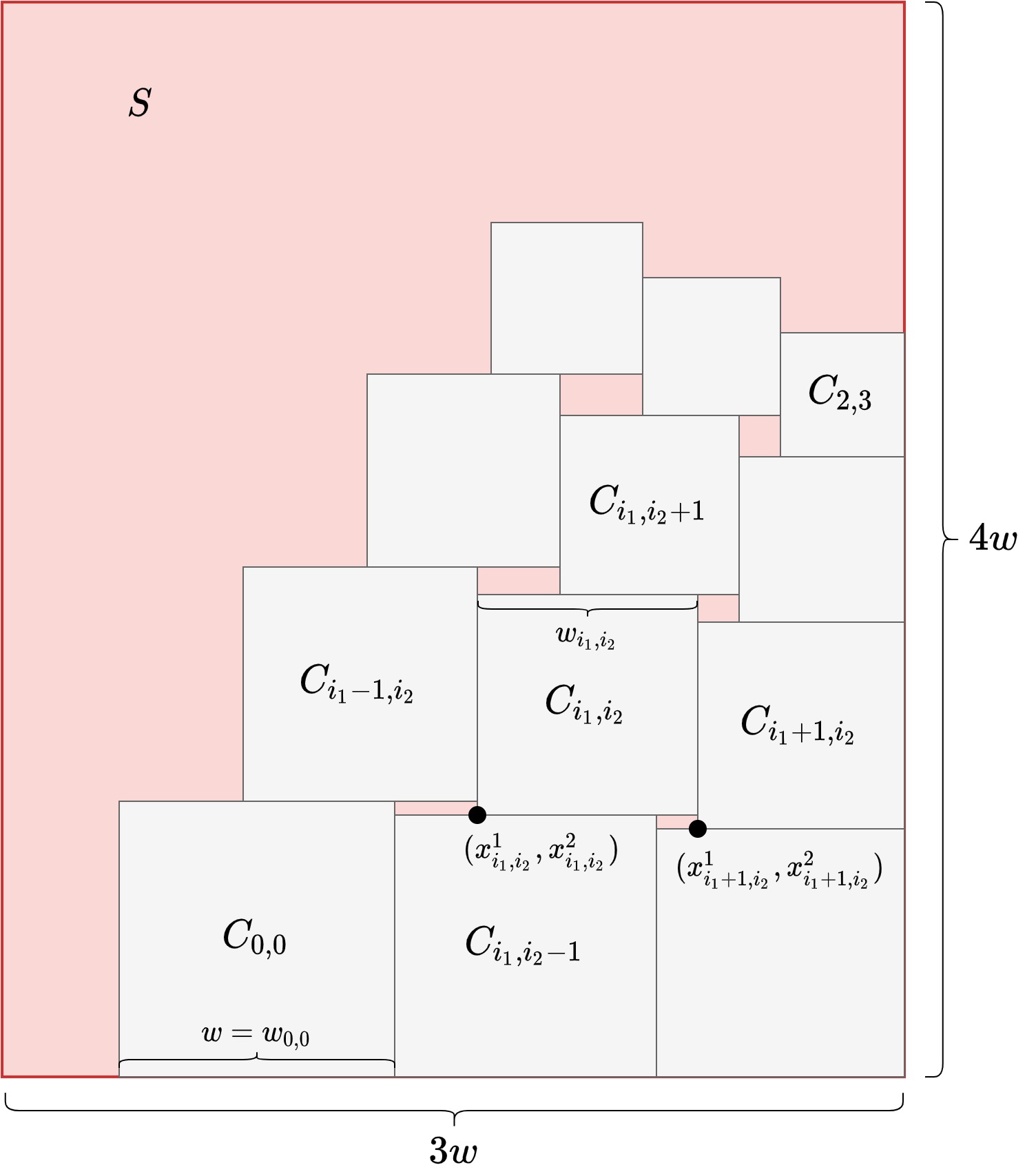}
    \centering
    \caption{The packing of the cubes $C_{\vec{i}} $ in $S$. Here, $d=2$, $M_1 = 3$, $M_2 = 4$, $M_* = 12$, and $i_1 = i_2 = 1$. Note that the diagram is not to scale.}
    \label{fig-C_i}
\end{figure}

We will verify that this is a legal packing shortly. Note that each $(x_{\vec{i}}^1, x_{\vec{i}}^2, \dots, x_{\vec{i}}^d)$ is asymptotically fixed at a lattice point as $\mathrm{sd}(\mathcal{C}) \ll \varepsilon$, namely
\begin{equation}\label{pack-lattice}
    (x_{\vec{i}}^1, x_{\vec{i}}^2, \dots, x_{\vec{i}}^d) = (wi_1,wi_2, \dots, wi_d) + O_M(\varepsilon w).
\end{equation}

Collect the subset of cubes from $\mathcal{C}$ which form its exterior "shell":
    \[\Tilde{\mathcal{C}} = \{C_{\vec{i}} \in \mathcal{C} : i_k =0 \text{ or } i_k = M_k - 1, \, \text{for some $k\in \{1, 2, \dots, d\}$}\}.\]
Let $B$ be the smallest brick containing $\mathcal{C} \setminus \Tilde{\mathcal{C}}$. By (\ref{pack-lattice}), $B$ has dimensions
    \[[w, (M_1-1) w] \times [w, (M_2-1) w] \times \dots \times [w, (M_d-1) w] + o(1).\]
    
\begin{figure}[t]
    \includegraphics[width=.8\textwidth]{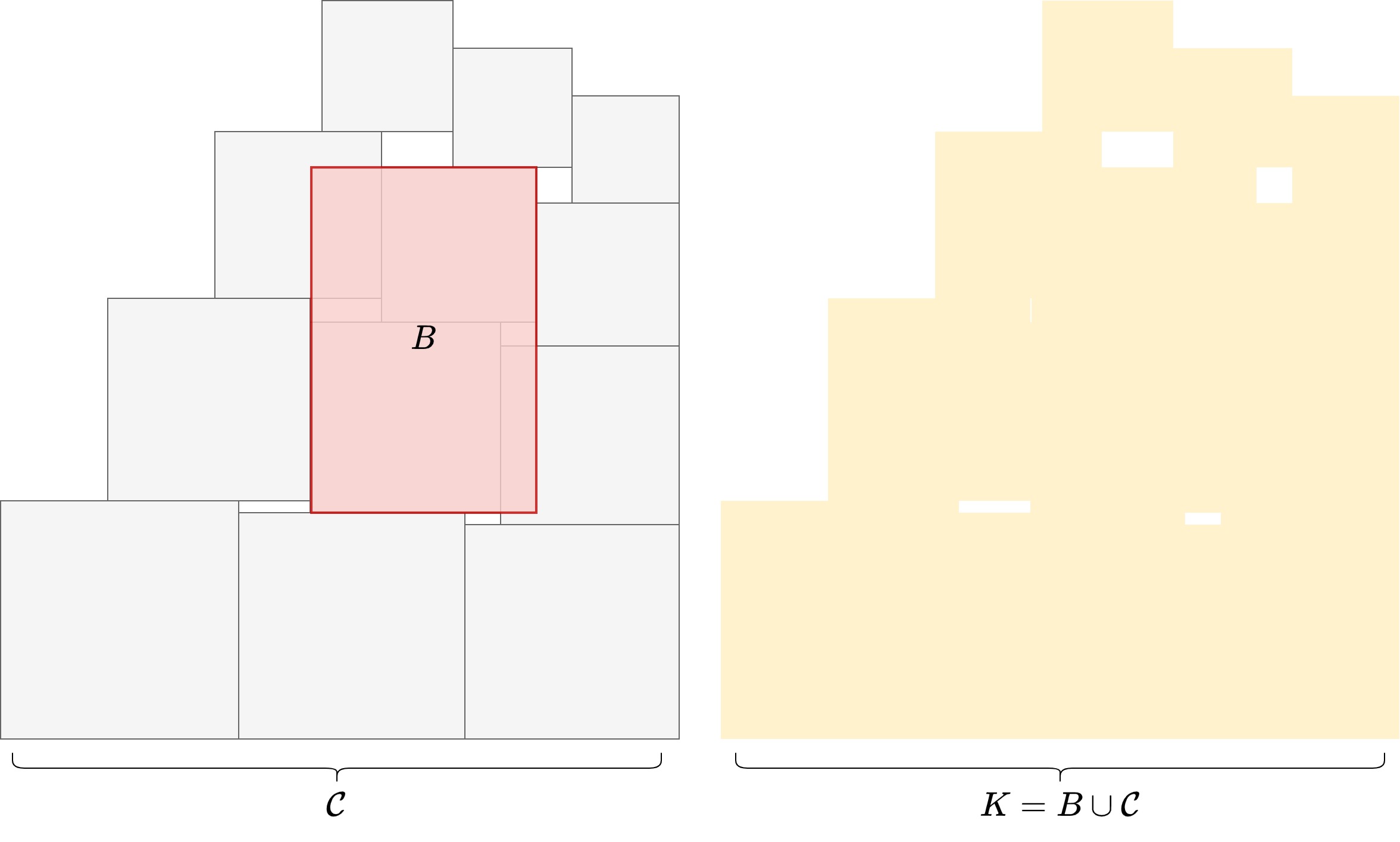}
    \centering
    \caption{The simple solid $K$ is constructed from $B$ and $\mathcal{C}$.}
    \label{fig-k}
\end{figure}

Define the simple solid $K =  B \cup \mathcal{C} = B \cup \Tilde{\mathcal{C}}$ (see Figure \ref{fig-k}). Observe that
    \[S \setminus K = (S \setminus B) \setminus \Tilde{\mathcal{C}}.\]
Observe that $S \setminus B$ can be partitioned into $O(M_*/ M)$ bricks $\mathcal{B}'$ each with dimensions $O(w)$. Each brick $B' \in \mathcal{B}'$ intersects at most $O(1)$ cubes in $\Tilde{\mathcal{C}}$. This means that we can partition $B' \setminus \Tilde{\mathcal{C}}$ into $O(1)$ bricks, each with weighted surface area less than that of $B'$, which is $\ll w^{d-1+\delta}$. Thus, $S \setminus K$ can be partitioned into bricks with allowable weighted surface area $\ll \frac{M_*}{M} w^{d-1+\delta}$. It then suffices to show that $K \setminus \mathcal{C}$ can be partitioned into bricks with weighted surface area $\ll \frac{M_*}{M} w^{d-1+\delta}$.

By Lemma \ref{r-snug}, it suffices to show that $\mathcal{C}$ is a packing that is $O_M(\varepsilon)$-snug in $K$. First, observe that all of the cubes are inside of $S$. This follows from the bound $w_{\vec{i}} \leq w$. Thus, we have to check that none of the cubes' interiors overlap, and that every cube in $\mathcal{C} \setminus \Tilde{\mathcal{C}}$ is touching the $2d$ adjacent cubes (the cubes in $\Tilde{\mathcal{C}}$ are already touching $K$ by construction).

Define $\pi_k$ to be the projection operator onto the $x_k$ axis for every $k \in \{1,2, \dots, d\}$. Let $\Tilde{E}$ be the collection of $2^d - 1$ vectors $\vec{e} = (e_1,e_2, \dots, e_d)$ such that every $e_k \in \{0,1\}$ but $\vec{e} \neq \vec{0}$. Let $\vec{e}_k$ be the $k$th unit vector, namely $(0,\dots,0,1,0,\dots, 0)$, with a $1$ in the $k$th component and let $E \subset \Tilde{E}$ be the collection of such $d$ unit vectors. Define $I$ be the collection of $\vec{i} = (i_1, i_2, \dots, i_d)$ such that $i_k \in \{0,1,\dots, M_k -2 \}$ for $k \in \{1,2, \dots, d\}$. By symmetry and the asymptotic positioning of the cubes (\ref{pack-lattice}), we only have to worry about checking overlap on "adjacent" cubes, reducing the proof to showing the following two claims:
\begin{enumerate}[(i)]
    \item Let $\vec{e} \in E$ and $\vec{i} \in I$. Then for at least one $k \in \{1, 2, \dots, d\}$ we have that $\pi_k(C_{\vec{i}}) \cap \pi_k(C_{\vec{i} + \vec{e}}) $ is \textit{exactly} one point.
    \item Let $\vec{e} \in \Tilde{E} \setminus E$ and $\vec{i} \in I$. Then for at least one $k \in \{1, 2, \dots, d\}$, we have that $\pi_k(C_{\vec{i}}) \cap \pi_k(C_{\vec{i} + \vec{e}}) $ is \textit{at most} one point.
\end{enumerate}
To see why this is sufficient to complete the proof, observe that as long as the boundary of the cubes are touching, the asymptotic positioning of the cubes (\ref{pack-lattice}) ensures that the non-overlapping boundary will have area $O_M(\varepsilon w)$, meaning that (i) will imply that the packing is $O_M(\varepsilon)$-snug. Clearly, (ii) implies that none of the cubes' interiors overlap, and thus our packing of $\mathcal{C}$ is valid.

To see (i), note that the construction of the $C_i$ immediately implies that for every $k \in \{1, 2, \dots, d\}$, we have
    \[\pi_k( C_{\vec{i} + \vec{e}_k}) \cap \pi_k (C_{\vec{i}}) = \{ x_{\vec{i}}^k + w_{\vec{i}}\}. \]

Now we show (ii). Fix some $\vec{e} = (e_1, e_2, \dots, e_d) \in \Tilde{E} \setminus E$. Let $k$ be the smallest index such that the component $e_k$ is nonzero. Clearly, $k \in \{1, 2, \dots, d-1\}$. Recall that
    \[x^{k}_{\vec{i}} = \sum_{i_{k}' = 0}^{M_{k} - 1} w_{i_1,\dots, i_{ k- 1}, i_{k}', 0, \dots, 0} - \sum_{i_{k}' = i_{k}} ^{M_{k}-1} w_{i_1,\dots, i_{k - 1},  i_{k}', i_{k} +1, \dots i_d}. \]
By the ordering of $w_{\vec{i}}$, we have that $w_{\vec{i} + \vec{e}} \leq w_{\vec{i} + \vec{e}_k}$. Thus, $x_{\vec{i} + \vec{e}} \geq x_{\vec{i} + \vec{e}_k} = x_{\vec{i}} + w_{\vec{i}}$. This shows that $\pi_k (C_{\vec{i} + \vec{e}}) \cap \pi_k (C_{\vec{i}})$ is at most a singleton, as desired. This completes the proof.
\end{proof}

\bibliography{references}

\end{document}